\documentclass{article}

\usepackage{mySymbols}
\usepackage{myPackages}
\title{The Real Vanishing Ideals of Nuclear $p$-Norm Balls}
\author{
  Ghislain Fourier\thanks{RWTH Aachen University, Germany. \texttt{fourier@art.rwth-aachen.de}} 
  \and 
  Yuhuai Zhou\thanks{RWTH Aachen University, Germany. \texttt{zhou@art.rwth-aachen.de}}
}
\date{}

\begin{document}

\maketitle

\begin{abstract}
We study the algebraic and geometric structure related to tensor nuclear norms. 
We show that the unit ball of the nuclear norm is the convex hull of an irreducible real variety and give an explicit description of its real vanishing ideal. 
As a consequence, we obtain a simple criterion to decide when a primary ideal is prime, 
and we use it to prove that the ideal of the nuclear $2$-norm is real reduced and prime.
\end{abstract}

\renewcommand{\thefootnote}{}


\section{Introduction}

The nuclear norm is a well-known and powerful tool in low-rank matrix recovery; 
see \cite{RauhutFoucart2013CompressiveSensing}, \cite{ChandrasekaranRecjtParrilo2012ConvexGeometryRecovery}, and \cite{Tropp2015ConvexRecovery}. 
It is natural to extend this idea to low-rank tensor recovery. 
Several authors have applied the nuclear norm to tensor completion problems; 
see \cite{YuanZhang2016TensorCompletion} and \cite{GandyRechtYamada2011TensorCompletion}. 
However, computing the nuclear norm for general tensors is NP-hard \cite{HillarLim2013NPTensor}. 
For this reason, different relaxations have been considered, 
for example \cite{RauhutStojanac2021ThetaNuclear}, \cite{RoehrichZhou2025ThetaNuclear}, and \cite{BarakMoitra2022SOSTensorCompletion}. 
Further background on tensor nuclear norms can be found in \cite{FriedlandLim2018NuclearNorm} and \cite{Derksen2016tOrthogonal}.

Our results go in two directions.  
First, we refine the algorithm of Krick and Logar by giving a simple criterion to decide when an ideal is prime.  
Second, we apply this to the family of nuclear $p$-norm ideals $I_p$, and show that for $p\in\{0,1,2,\infty\}$ these ideals are real radical.  
In particular, we prove that $I_2$ is real reduced and prime.

Let $d\in\N_+$ be the order of a tensor, and $\nbf = (n_1,\dots,n_d)\in\N_+^d$ its size.  
For $n\in\N_+$, we write $[n]=\{1,\dots,n\}$, and $[\nbf]=[n_1]\times\cdots\times[n_d]$ for the index set.

The works \cite{RauhutStojanac2021ThetaNuclear} and \cite{RoehrichZhou2025ThetaNuclear} 
describe the unit ball of the nuclear norm as the convex hull of a real algebraic variety,
\begin{equation}\label{eq: intro-nuclear-ideal}
    B_2 = \conv(\Vcal_\R(I_2)), \qquad 
    I_2 = \langle \sum_{a\in[\nbf]} x_a^2 - 1,\; \text{rank-1 binomials} \rangle.
\end{equation}
The explicit form of the rank-1 binomials will be recalled later.  
For general even numbers $p=2s$, this extends to
\begin{equation}\label{eq: intro-nuclear-p-ideal}
    B_{2s} = \conv(\Vcal_\R(I_{2s})), \qquad 
    I_{2s} = \langle \sum_{a\in[\nbf]} x_a^{2s} - 1,\; \text{rank-1 binomials} \rangle.
\end{equation}
Moreover, \cite{RoehrichZhou2025ThetaNuclear} gives similar descriptions for $p=1$ and $p=\infty$:
\begin{equation}\label{eq: intro-nuclear-1-ideal}
    I_1 = \langle \sum_{a\in[\nbf]} x_a^2 - 1,\, x_a^3 - x_a,\, x_a x_b \; (a\neq b)\rangle,
    \qquad
    I_\infty = \langle x_a^2 - 1\; (a\in[\nbf]),\, \text{rank-1 binomials}\rangle.
\end{equation}

These constructions are consistent with the general definition of nuclear $p$-norms.  
For $p=1$, the result is simply the $\ell_1$-norm (see \cite{FriedlandLim2018NuclearNorm, ChenLi2020NuclearPNorm, RoehrichZhou2025ThetaNuclear}).  
Throughout the paper, by the nuclear norm we always mean the nuclear $2$-norm.  
For completeness, we also introduce the ideal
\begin{equation}\label{eq: intro-rank-1-ideal}
    I_0 := \langle \text{rank-1 binomials} \rangle,
\end{equation}
which defines the Segre variety.  
Although the convex hull of $\Vcal_\R(I_0)$ is not a proper norm, the ideal $I_0$ remains important.  
It describes all rank-one tensors, and its reducedness will be used later to show that $I_2$ is real reduced.

This work provides a first systematic study of tensor nuclear $p$-norms from the viewpoint of real algebraic geometry.  
We believe that such algebraic information gives new insight into convex relaxations, such as the theta norms introduced in \cite{GouveiaParriloThomas2010ThetaBody}.  
We prove that $I_p$ is real reduced for $p=0,1,2,\infty$, hence equal to the real vanishing ideal of its real variety.  
In addition, $I_2$ is irreducible and prime.  
For general even $p=2s$, one can verify with computer algebra systems such as OSCAR \cite{OSCAR} that $I_{2s}$ is prime; in that case, our arguments also apply to show that $I_{2s}$ is real reduced.

Our first main result gives a practical criterion for deciding if an ideal is prime.

\begin{theorem}\label{thm: intro-primary2prime}
Let $\Kbb$ be a perfect field and $I\subset\Kbb[x_1,\dots,x_n] =: \Kbb[X]$ an ideal.  
Let $\Jcal\subseteq[n]$ be a maximal algebraically independent subset in $\Kbb[X]/I$, and let $k=\C(\Jcal)$ be the field of rational functions in $\Kbb[\Jcal]:= \Kbb[x_a: a\in\Jcal]$.  
If for every $a\notin\Jcal$ there exists $H_a\in I\cap\Kbb[\Jcal\cup\{a\}]$ such that $H_a(x_a)\in k[x_a]$ is square-free, then $I$ is prime.
\end{theorem}

We then apply this result to the nuclear $2$-norm ideal.

\begin{theorem}\label{thm: intro-realprime}
The unit ball of the nuclear $2$-norm is the convex hull of an irreducible real variety 
defined by the real reduced prime ideal $I_2$ in \eqref{eq: intro-nuclear-ideal}.  
For $p\in\{0,1,\infty\}$, the corresponding ideals $I_p$ in \eqref{eq: intro-nuclear-1-ideal} are also real reduced.
\end{theorem}

These results establish a clear link between convex geometry and real algebraic geometry in the study of tensor nuclear norms.  
They also provide algebraic tools that can be used for further developments, such as theta-body relaxations and related convex formulations.

\paragraph{Structure of the paper.}
Section~\ref{section: preliminary} fixes notation and recalls background on (real) radicals, complexification, and algebraic independence.
In Section~\ref{section: p-0-1-infty} we take a first look at the cases where $p \in \set{0,1,\infty}$ and prove that such ideals $I_p$ are real radical.
In Section~\ref{section: primary-2-prime} we prove Theorem~\ref{thm: intro-primary2prime}, which provides a criterion for showing that a primary ideal is in fact prime over perfect base fields.
Section~\ref{section: real prime} applies this framework to the nuclear $p$-norm ideals for even numbers $p=2s$.
We show that $I_2$ is a real reduced prime.
For even $p=2s$, assuming that $I_{2s}$ is primary (which can be verified \emph{a priori} for fixed tensor sizes in a computer algebra system such as OSCAR), our method further proves that $I_{2s}$ is prime.

\numberwithin{theorem}{section}
\numberwithin{equation}{section}

\section{Preliminaries and Notation}\label{section: preliminary}

\subsection{Vanishing Ideals and the Nullstellensatz}\label{section: pre-nullstellensatz}

Let $\Kbb$ be a field.

\begin{definition}
    Let $I \subseteq \Kbb[X]$ be an ideal. 
    Its (complex) radical is defined by
    \[
        \sqrt{I} := \set{f \in \Kbb[X] : f^r \in I \text{ for some } r > 0}.
    \]
    When $\Kbb = \R$, the \emph{real radical} of $I$ is
    \[
        \Rradical{I} := 
        \set{
            f \in \R[X] :
            f^{2r} + \sum_{i=1}^l g_i^2 \in I
            \text{ for some } r > 0 \text{ and } g_i \in \R[X]
        }.
    \]
    We say that $I$ is \emph{reduced} (or \emph{radical}) if $I = \sqrt{I}$,
    and \emph{real reduced} (or \emph{real radical}) if $I = \Rradical{I}$.
    The ideal $I$ is called \emph{real} if $-1$ cannot be expressed
    as a sum of squares modulo~$I$.
\end{definition}

For a subset $V \subset \Kbb^n$, we denote by $I(V) \subset \Kbb[X]$
the ideal of all polynomials that vanish on $V$.

\begin{theorem}[Nullstellensatz]
    If $\Kbb = \C$, then
    \[
        I(\Vcal_\C(I)) = \sqrt{I}
        \qquad\text{and}\qquad
        \sqrt{\set{0}} = \set{0}.
    \]
    If $\Kbb = \R$, then
    \[
        I(\Vcal_\R(I)) = \Rradical{I}
        \qquad\text{and}\qquad
        \Rradical{\set{0}} = \set{0}.
    \]
    Moreover, an ideal $I \subseteq \R[X]$ is real if and only if
    its real variety $\Vcal_\R(I)$ is nonempty.
\end{theorem}

There are three main approaches to check whether an ideal is real radical:

\begin{enumerate}
    \item \textbf{Complexification test.}
    If $\Vcal_\C(\Vcal_\R(I)) = \Vcal_\C(I)$,
    then one only needs to verify that $I$ is (complex) radical.

    \item \textbf{Sum-of-squares criterion.}
    As observed in \cite{Marshall2008PositiveSOS},
    it suffices to check the implication
    \[
        f_1^2 + \cdots + f_l^2 \in I \ \Longrightarrow\ f_i \in I
        \quad\text{for all } i.
    \]

    \item \textbf{Single-point criterion.}
    If $I$ is prime, it is enough to find a smooth real point on the variety
    $\Vcal_\R(I)$; see Section~\ref{section: real prime} and
    \cite{Marshall2008PositiveSOS}.
\end{enumerate}

In this work we employ all three methods to verify the real reducedness
of the ideals~$I_p$.
For further background on real algebraic geometry, we refer to
\cite{Marshall2008PositiveSOS,KnebuschScheiderer2022RealAlg}.

\subsection{Complexification of Real Ideals}\label{section: complexification}

To circumvent certain subtleties of real algebraic geometry,
we will often transfer problems to the complex setting.
Let $I \subseteq \R[X]$ be an ideal, and define its \emph{complexification} as
\[
    I_\C := I\,\C[X].
\]
When the ambient field is clear from context, we use the simplified notation
\[
    \Vcal_\R(I) = \Vcal(I)
    \quad\text{and}\quad
    \Vcal_\C(I) = \Vcal_\C(I_\C).
\]
The next lemma shows that complexification behaves well with respect to
primary and prime ideals.

\begin{lemma}
    Let $I \subseteq \R[X]$ be an ideal and let $I_\C$ be its complexification.
    Then
    \[
        I_\C = I\,\C[X]
        \quad\text{and}\quad
        I = I_\C \cap \R[X].
    \]
    Moreover, if $I_\C$ is primary (respectively, prime),
    then $I$ is also primary (respectively, prime).
\end{lemma}

\begin{proof}
    The equality $I_\C = I\,\C[X]$ holds by definition.
    Let $f \in I_\C \cap \R[X]$.
    Then $f = \sum_j (a_j + ib_j)h_j$ with $a_j,b_j \in \R[X]$ and $h_j \in I$.
    Taking complex conjugates and adding gives a real expression
    \[
        f = \sum_j a_j h_j \in I.
    \]
    The converse inclusion is obvious, hence $I = I_\C \cap \R[X]$.

    Now assume that $I_\C$ is primary.
    If $f,g \in \R[X]$ satisfy $fg \in I \subseteq I_\C$,
    then either $f \in I_\C$ or $g^r \in I_\C$ for some $r>0$.
    Since both $f$ and $g$ are real, it follows that $f \in I$ or $g^r \in I$,
    so $I$ is primary.
    The argument for prime ideals is analogous.
\end{proof}

We next show that complexification preserves the dimension of an ideal.
There are several equivalent notions of dimension;
we adopt the one based on algebraically independent variable subsets
(see \cite{Kemper2011CommAlg}).

Let $\R[X] = \R[x_1,\dots,x_n]$ and $\C[X] = \C[x_1,\dots,x_n]$.
For $\Jcal \subseteq [n] := \set{1,\dots,n}$, define
$\R[\Jcal] := \R[x_i : i \in \Jcal]$.
We say that $\Jcal$ is \emph{algebraically independent} (alg.~ind.) with respect to $I$
if $\R[\Jcal] \cap I = \set{0}$.
The \emph{dimension} of $I$ is then defined as
\[
    \dim(I) := 
    \max\set{|\Jcal| : \Jcal \subseteq [n], \ \R[\Jcal] \cap I = \set{0}}.
\]

\begin{lemma}
    Let $I \subseteq \R[x_1,\dots,x_n]$ be an ideal.
    A subset $\Jcal \subseteq [n]$ is algebraically independent with respect to $I$
    if and only if it is algebraically independent with respect to $I_\C$
    in $\C[X]$. In particular, $\dim(I) = \dim(I_\C)$.
\end{lemma}

\begin{proof}
    It suffices to consider algebraic dependence.
    If $\Jcal$ is algebraically dependent with respect to $I$,
    there exists a nonzero $f \in I \cap \R[\Jcal]$.
    Then $f \in I_\C \cap \C[\Jcal]$, so $\Jcal$ is also dependent for $I_\C$.

    Conversely, suppose $\Jcal$ is dependent for $I_\C$.
    Let $f = f_1 + i f_2 \in I_\C \cap \C[\Jcal]$ be a nonzero dependence polynomial,
    with $f_1,f_2 \in \R[\Jcal]$.
    Both $f_1$ and $f_2$ belong to $I\cap \R[\Jcal]$,
    showing that $\Jcal$ is also dependent for $I$.
\end{proof}

Finally, if $I$ is primary, then all maximal algebraically independent subsets
have the same cardinality.
Equivalently, the algebraic independence with respect to~$I$
forms a matroid (see \cite{Oxley1992Matroid}).

\begin{lemma}
    Let $\Kbb$ be a field and $I \subseteq \Kbb[X_1,\dots,X_n]$ a primary ideal.
    Then the algebraic independence of $\set{X_1,\dots,X_n}$ with respect to $I$
    defines an algebraic matroid.
\end{lemma}

\begin{proof}
    Algebraic independence with respect to~$I$ coincides with that
    with respect to its radical $P = \sqrt{I}$.
    Since $P$ is prime, it corresponds to the algebraic independence
    of elements in the field extension
    \[
        \Kbb \hookrightarrow \mathrm{Frac}(\Kbb[X]/P).
    \]
    Such independence defines an algebraic matroid
    (see \cite{Oxley1992Matroid}).
\end{proof}

\section{A First Look at the Cases \texorpdfstring{$p = 0, 1, \infty$}{}}\label{section: p-0-1-infty}

In this section we establish the real reducedness of the ideals $I_p$ for
$p \in \{0,1,\infty\}$.
We employ the first two techniques introduced in Section~\ref{section: pre-nullstellensatz}, namely the complexification test and the sum-of-squares criterion.

Our main computational tool is the Gr\"obner basis computed in \cite{RoehrichZhou2025ThetaNuclear}. For convenience, we recall the underlying monomial order. For variables $a,b \in [\nbf]$ we write $x_a > x_b$ whenever $a<b$ in lexicographic order. For monomials $x^\alpha$ and $x^\beta$ with exponent vectors $\alpha,\beta \in \N^\nbf$, we use the graded reverse lexicographic (grevlex) order:
\[
    x^\alpha > x^\beta
    \iff
    \big(\sum_i \alpha_i > \sum_i \beta_i\big)
    \text{ or }
    \big(\sum_i \alpha_i = \sum_i \beta_i
    \text{ and the rightmost nonzero entry of }\alpha-\beta\text{ is negative}\big).
\]
Let $\Gcal_p$ denote the Gr\"obner basis of $I_p$ (and of its complexification $I_{p,\C}$). Explicit descriptions of $\Gcal_p$ will be given only where needed. The key fact for our purposes is that
\[
    \Gcal_p \subseteq \R[X],
\]
that is, all elements of the Gr\"obner basis are real polynomials.

\subsection*{The Cases \texorpdfstring{$p = 1$}{p=1} and \texorpdfstring{$p = \infty$}{p=infty}}

For $p = 1$ and $p = \infty$, we employ the first approach from Section~\ref{section: pre-nullstellensatz}. In both cases, the real and complex varieties coincide,
\[
    \Vcal_\R(I_p) = \Vcal_\C(I_p),
\]
and are finite. Hence $\dim I_p = 0$, and the question of real reducedness reduces to  verifying that $I_p$ is (complex) radical. This will follow from Seidenberg's lemma
(\cite[Proposition~3.7.15]{KreuzerRobbiano2008CompCommAlg1}).

\begin{lemma}[Seidenberg's Lemma]
    Let $\Kbb$ be a field and $I \subseteq \Kbb[x_1,\dots,x_n]$ a zero-dimensional ideal.     Suppose that for every $i \in \{1,\dots,n\}$ there exists a nonzero polynomial $f_i \in I \cap \Kbb[x_i]$ such that $f_i$ and its derivative $f_i'$ are coprime in $\Kbb[x_i]$.
    Then $I$ is radical.
\end{lemma}

Recall that
\begin{equation}\label{eq:I1}
    I_1 = \gen{
        \sum_{a \in [\nbf]} x_a^2 - 1,\,
        x_a^3 - x_a \text{ for all } a\in[\nbf],\,
        x_a x_b \text{ for } a \neq b \in [\nbf]
    }
\end{equation}
and
\begin{equation}\label{eq:Iinfty}
    I_\infty = \gen{ x_a^{2}-1 \text{ for } a\in[\nbf],\,  \text{rank-1 binomials} }.
\end{equation}
The ideal $I_1$ coincides with its Gr\"obner basis, while $I_\infty$ admits a corresponding basis described in
\cite{RoehrichZhou2025ThetaNuclear}.

\begin{corollary}
    Let $p \in \{1,\infty\}$ and let $I_{p,\C}$ denote the complexification of $I_p$. Then $I_{p,\C}$ is radical.
\end{corollary}

\begin{proof}
    For both $p = 1$ and $p = \infty$ the variety     $\Vcal_\C(I_p) = \Vcal_\R(I_p)$ is finite, so $I_{p,\C}$ is zero-dimensional. For each index $a$ set
    \[
        f_a(x_a) =
        \begin{cases}
            x_a^3 - x_a, & p = 1,\\
            x_a^2 - 1,   & p = \infty.
        \end{cases}
    \]
    Then $f_a, f_a'$ are coprime in $\C[x_a]$, and Seidenberg's lemma implies that $I_{p,\C}$ is radical.
\end{proof}

\begin{theorem}
    For $p \in \{1,\infty\}$, the ideal $I_p \subseteq \R[X]$ is real radical.
\end{theorem}

\begin{proof}
    Since $\Vcal_\C(I_{p,\C}) = \Vcal_\R(I_p)$, Hilbert’s Nullstellensatz yields that any
    $f \in \R[X]$ vanishing on $\Vcal_\R(I_p)$ belongs to $I_{p,\C}$. Using the Gr\"obner basis $\Gcal_p$ from \cite{RoehrichZhou2025ThetaNuclear}, the division algorithm gives an explicit representation
    \[
        f(x) = \sum_{g(x)\in\Gcal_p} c_g(x)\, g(x),
    \]
    with all coefficients $c_g(x)$ and basis elements $g(x)$ real. Hence $f \in I_p \subseteq \R[X]$. Thus $I_p$ is precisely the vanishing ideal of its real variety, and by the Real Nullstellensatz it follows that $I_p$ is real radical.
\end{proof}

\subsection*{The Case \texorpdfstring{$p = 0$}{p=0}}

As mentioned earlier, the reducedness of $I_0$ will play a key role in establishing the real reducedness of $I_2$. We present the argument here as a first illustration. Recall that $I_0$ is the ideal generated by all rank-one binomials. To describe these binomials explicitly, let $a,b\in[\nbf]$ be distinct multi-indices, and define
\[
a\wedge b := (\min\{a_1,b_1\},\dots,\min\{a_d,b_d\}), \qquad a\vee b := (\max\{a_1,b_1\},\dots,\max\{a_d,b_d\}).
\]
Following \cite{RoehrichZhou2025ThetaNuclear} and \cite{RauhutStojanac2021ThetaNuclear}, the rank-one condition is characterized by all binomials of the form
\[
x_a x_b - x_{a\vee b} x_{a\wedge b}, \qquad a,b\in[\nbf].
\]
To avoid trivial polynomials, we define
\[
\Scal := \{(a,b)\in[\nbf]^2 : \exists\, i,j\in[d] \text{ such that } a_i>b_i,\ a_j<b_j\}.
\]
Then
\begin{equation}\label{def: I0}
    I_0 = \big\langle\, x_a x_b - x_{a\vee b}x_{a\wedge b} : (a,b)\in\Scal \,\big\rangle \subseteq \R[X].
\end{equation}
Moreover, the set of generators above forms a Gr\"obner basis whose leading terms are $x_a x_b$ with $(a,b)\in\Scal$. Let $I_{0,\C}$ denote the complexification of $I_0$. 
We now show that $I_0$ is real reduced and that $I_{0,\C}$ is reduced.

\begin{theorem}\label{thm: reduced-I0}
The real ideal $I_0$ is real radical, and its complexification $I_{0,\C}$ is radical.
\end{theorem}

\begin{proof}
It suffices to use the Gr\"obner basis $\Gcal_0 = \{x_a x_b - x_{a\vee b}x_{a\wedge b} : (a,b)\in\Scal\}$. Since the leading terms $x_a x_b$ are square-free, both $I_0$ and $I_{0,\C}$ are reduced; see \cite{CoxLittleOShea1997AlgGeometry}. 
To establish real reducedness, we apply the second criterion from Section~\ref{section: pre-nullstellensatz}. Let $\LT(\cdot)$ denote the leading term. 
Suppose $f^2\in I_0$ but $f\notin I_0$. Then the remainder $r$ of $f$ upon division by $\Gcal_0$ satisfies $x_a x_b \nmid \LT(r)$ for all $(a,b)\in\Scal$. 
However, $f^2\in I_0$ implies that $r^2\in I_0$, so there exists $(a,b)\in\Scal$ with $x_a x_b\mid \LT(r^2)$, hence $x_a x_b\mid \LT(r)$ — a contradiction. 

Now consider a sum of squares $\sum_{i=1}^l f_i^2 \in I_0$. We show that each $f_i\in I_0$. Let $r_i$ denote the remainder of $f_i$ upon division by $\Gcal_0$, so $\sum_i r_i^2\in I_0$. Without loss of generality, order the leading terms as $\LT(r_1)>\LT(r_2)>\cdots>\LT(r_l)$. Then some $(a,b)\in\Scal$ satisfies $x_a x_b\mid \LT(\sum_i r_i^2)$, implying $x_a x_b\mid \LT(r_1)$. 
If $r_1\neq 0$, this is again a contradiction; thus $r_1=0$ and hence all $r_i=0$. 
Therefore $f_i\in I_0$ for all $i$, and by \cite[Proposition~12.5.1]{Marshall2008PositiveSOS}, the ideal $I_0$ is real radical.
\end{proof}

\section{From Primary to Prime Ideals}\label{section: primary-2-prime}

Before turning to the real reducedness of $I_{2s}$, we establish our first main tool, Theorem~\ref{thm: intro-primary2prime}. It refines the algorithm of \cite{KrickLogar1991Radical} for the special case of primary ideals and serves as a key ingredient in the proof of Theorem~\ref{thm: intro-realprime}. In practice, for any fixed tensor size $\nbf$ and even exponent $p = 2s$, one can verify—using algebra systems such as OSCAR~\cite{OSCAR}—that $I_{2s}$ is primary or even prime. The goal of this section is to address the behavior of such primary ideals and to prove Theorem~\ref{thm: intro-primary2prime}.

Recall that Seidenberg's lemma provides a practical criterion for testing the reducedness of zero-dimensional ideals.

\begin{lemma}[Seidenberg's Lemma]
Let $\Kbb$ be a field and $I \subseteq \Kbb[x_1,\dots,x_n]$ a zero-dimensional ideal. Suppose that for each $i \in \{1,\dots,n\}$ there exists a nonzero polynomial $f_i \in I \cap \Kbb[x_i]$ such that $f_i$ and its derivative $f_i'$ are coprime in $\Kbb[x_i]$. Then $I$ is radical.
\end{lemma}

\begin{remark}
A polynomial $f \in \Kbb[X]$ is called \emph{strongly square-free} if $f$ and $f'$ are coprime in $\Kbb[X]$. By \cite[Proposition~3.7.9]{KreuzerRobbiano2008CompCommAlg1}, this condition implies that $f$ is square-free, and the converse holds whenever $\Kbb$ is perfect. In practice, the condition can be verified by computing $\gcd(f,f') = 1$.
\end{remark}

More generally, \cite{KrickLogar1991Radical} presented an algorithm to compute radicals of arbitrary ideals over perfect fields. For primary ideals, however, this procedure can be simplified: it is no longer necessary to determine a Noether position or to iterate through all algebraically independent subsets. The result below formalizes this simplification. While we formulate it for $\R$ and $\C$, it holds for all perfect fields.

Let $I \subseteq \R[x_1,\dots,x_n]$ be a primary ideal, and let $I_\C$ denote its complexification. Let $\Jcal \subseteq [n] := \{1,\dots,n\}$ be a maximal algebraically independent subset with respect to $I$, i.e.\ $I_\C \cap \C[\Jcal] = \{0\}$ and $\dim(I) = |\Jcal|$. Then the following theorem holds.

\begin{theorem}\label{thm: primary-2-prime}
Let notation be as above and define $k := \C(\Jcal)$, the field of fractions of $\C[\Jcal]$. If for every $a \in \Jcal^c$ there exists a polynomial $H_a \in I_\C \cap \C[\Jcal \cup \{a\}]$ such that $H_a(x_a), H_a'(x_a) \in k[x_a]$ are coprime, then $I_\C$ is prime. In particular, $I \subseteq \R[X]$ is also prime.
\end{theorem}

\begin{proof}
Let $R := k[\Jcal^c]$ and extend $I_\C$ to $R$ by setting $I_\C^E := R \cdot I_\C$. Since $\Jcal$ is a maximal algebraically independent subset, $\dim(I_\C^E) = 0$ in $R$. By Seidenberg's lemma, $I_\C^E$ is reduced, because for each $a \in \Jcal^c$ the polynomials $H_a, H_a' \in k[x_a]$ are coprime.

We claim that $I_\C^E \cap \C[X] = I_\C$. The inclusion $I_\C \subseteq I_\C^E \cap \C[X]$ is obvious. Conversely, if $f \in I_\C^E \cap \C[X] = (k[\Jcal^c] \cdot I_\C) \cap \C[X]$, then $f = p g$ for some $p \in k[\Jcal^c]$ and $g \in I_\C$. Write $p = p_1 / p_0$ with nonzero $p_0 \in \C[\Jcal]$ and $p_1 \in \C[X]$. Then $f p_0 = p_1 g \in I_\C$. Since $I_\C \cap \C[\Jcal] = \{0\}$, no power of $p_0$ lies in $I_\C$, and because $I_\C$ is primary, it follows that $f \in I_\C$.

Now suppose $f^r \in I_\C$. As $I_\C^E$ is reduced, $f \in I_\C^E$ and hence $f \in I_\C$. Thus $I_\C$ is both primary and reduced, which implies that it is prime. Consequently, $I \subseteq \R[X]$ is prime as well.
\end{proof}

\section{The Vanishing Ideal of the Nuclear Norm}\label{section: real prime}

Recall that for $p = 2s$, the unit ball of the nuclear $p$-norm can be expressed as the convex hull of the real variety of the ideal
\begin{equation*}
    I_p := \big\langle\, \sum_a x_a^p - 1,\ \text{rank-1 binomials} \,\big\rangle.
\end{equation*}
In this section we complete the proof of our second main result, Theorem~\ref{thm: intro-realprime}, showing that $I_2$ is a real reduced prime.

Although we cannot, in general, prove that $I_p$ is primary for all even $p = 2s$, this property can be verified computationally for any fixed tensor size using algebra systems such as \textsc{OSCAR}~\cite{OSCAR}. Assuming that $I_p$ is primary, our method extends to all such cases, implying that each $I_p$ is a real reduced prime whose real variety is irreducible.

The structure of this section is as follows.  
First, we recall the construction and symmetry properties of the nuclear $2$-norm unit ball and the corresponding ideal $I_2$.  
Second, we show that $I_2$ is a primary ideal.  
Next, under the assumption that $I_p$ is primary for all $p = 2s$, we use the result from Section~\ref{section: primary-2-prime} to conclude that each $I_p$ is prime.  
Finally, we establish real reducedness through a smoothness argument.  
In particular, we prove that $I_2$ coincides with the real vanishing ideal of its variety, which is of central interest in tensor applications.

Throughout this section we write $I_{p,\C}$ for the complexification of $I_p$.

\subsection{Construction and Symmetries of the Ideal \texorpdfstring{$I_2$}{}}\label{section: construction-symmetry-I2}

The unit ball of the nuclear norm is the convex hull of all unit rank-one tensors. The unit constraint is expressed by the polynomial
\[
    \sum_{a \in [\nbf]} x_a^2 - 1.
\]
Recall that the rank-one condition is captured by the ideal
\[
    I_0 = \big\langle\, x_a x_b - x_{a\vee b} x_{a\wedge b} : (a,b) \in \Scal \,\big\rangle,
\]
where
\[
    \Scal := \{ (a,b) \in [\nbf]^2 : \exists\, i,j \in [d] \text{ with } a_i > b_i,\ a_j < b_j \}.
\]
Hence,
\begin{equation}\label{eq: nuclear-ideal}
    I_2 = \big\langle\, \sum_{a\in[\nbf]} x_a^2 - 1 \,\big\rangle + I_0.
\end{equation}
The corresponding nuclear norm ball is
\[
    B_* = \conv\big(\Vcal_\R(I_2)\big),
\]
and the generating set in \eqref{eq: nuclear-ideal} also forms a Gr\"obner basis.

\paragraph{Symmetries.} The symmetries of the nuclear norm have been studied in several works;
see \cite{RoehrichZhou2025ThetaNuclear} and \cite{FriedlandLim2018NuclearNorm}. 
Let
\[
    G := \SOgroup(n_1) \times \cdots \times \SOgroup(n_d),
\]
which acts naturally on the tensor space $\R^\nbf$ by
\[
    (\tensorvec{U}{1} \times \cdots \times \tensorvec{U}{d}) \cdot (\puretensor{v}{d})
    := \tensorvec{U}{1} \tensorvec{v}{1} \otimes \cdots \otimes \tensorvec{U}{d} \tensorvec{v}{d}.
\]
It is immediate that the real variety $\Vcal_\R(I_2)$ is invariant under this action.

For the complex setting, define the complex special orthogonal group
\[
    \SOgroup(n,\C) := \{\, U \in \C^{n\times n} : U U^{T} = \Id \,\},
\]
and let
\[
    G_\C := \SOgroup(n_1,\C) \times \cdots \times \SOgroup(n_d,\C).
\]
Then the complex variety $\Vcal_\C(I_{2,\C})$ is precisely the $G_\C$-orbit of the tensor $e_1 \otimes \cdots \otimes e_1$, i.e.
\[
    \Vcal_\C(I_{2,\C}) = G_\C \cdot (e_1 \otimes \cdots \otimes e_1)
    = \big\{\, \puretensor{u}{d} : (\tensorvec{u}{i})^{T} \tensorvec{u}{i} = 1 \,\big\}.
\]
We next show that both the ideal $I_{2,\C}$ and its real counterpart $I_2$ are invariant under the corresponding group actions.

\begin{proposition}
Let the actions of $G$ and $G_\C$ be as above. Then both the ideals and their varieties are invariant under these actions:
\[
    g\cdot f(x) := f(g^{-1}x) \in I_{2,\C} \quad \forall\, f\in I_{2,\C},\ g\in G_\C,
\]
and similarly for the real case. Moreover, the actions of $G$ and $G_\C$ on their respective varieties are transitive.
\end{proposition}

\begin{proof}
The invariance of the varieties follows directly from the definitions of the actions. To verify the invariance of $I_2$, note first from Theorem~\ref{thm: reduced-I0} that $I_0$ is real reduced and hence the vanishing ideal of the set of real rank-one tensors. This variety is $G$-invariant, and by the Real Nullstellensatz, so is $I_0$. Since $G$ is a subgroup of the orthogonal group, the polynomial $\sum_{a\in[\nbf]} x_a^2 - 1$ is fixed under its action; hence $\langle\, \sum_a x_a^2 - 1 \,\rangle$ is $G$-invariant as well. As the sum of two $G$-invariant ideals, $I_2$ is $G$-invariant. The same reasoning applies to $G_\C$ and $I_{2,\C}$, completing the proof.
\end{proof}

\subsection{The Ideal \texorpdfstring{$I_2$}{} is Primary}\label{section: I2-primary}

We begin with a general structural result for algebraically closed fields concerning transitive varieties and their invariant ideals.

\begin{theorem}
Let $\Kbb$ be an algebraically closed field, and let $G$ be a group acting on the affine space $\Kbb^n$. Suppose $V \subseteq \Kbb^n$ is a $G$-transitive variety defined by a $G$-invariant ideal $I \subset \Kbb[X]$, i.e.\ $V = \Vcal_\Kbb(I)$. If $V$ is connected in the Euclidean topology, then $V$ is irreducible and $I$ is primary.
\end{theorem}

\begin{proof}
Assume that $V = V_1 \cup \cdots \cup V_k$ is the minimal irreducible decomposition of $V$, and suppose $k > 1$. Choose points $x_i \in V_i$ such that $x_i \notin \bigcup_{j \neq i} V_j$. By $G$-transitivity,  each $x_i$ lies in exactly one irreducible component $V_i$, 
and hence every point of $V$ belongs to exactly one irreducible component. 
Hence the decomposition is disjoint, implying that $V$ is disconnected in the Zariski topology and therefore also disconnected in the Euclidean topology—a contradiction. Thus $V$ must be irreducible.

To show that $I$ is primary, recall that this is equivalent to $\Kbb[X]/I$ having a unique associated prime \cite{AtiyahMacdonald1969CommAlg}. Let $P$ be an associated prime of $\Kbb[X]/I$ with respect to $f$, so that $Pf \subset I$. For any $r \in G$, the ideal $r \cdot P$ is an associated prime corresponding to $r \cdot f$, since $I$ is $G$-invariant. Denote $V_P = \Vcal_\Kbb(P)$, an irreducible subvariety of $V$.  

Pick any $x \in V_P$. By transitivity, for every $y \in V$, there exists $r_y \in G$ such that $r_y \cdot x = y$, implying that $y \in r_y \cdot V_P = V_{r_y \cdot P}$. Hence
\[
V = \bigcup_{r \in G} r \cdot V_P.
\]
Since $\Kbb[X]/I$ has only finitely many associated primes, this union is finite. The irreducibility of $V$ then forces $P$ to be $G$-invariant and $V_P = V$. Consequently, $\Kbb[X]/I$ has exactly one associated prime, namely $\sqrt{I}$, and therefore $I$ is primary.
\end{proof}

Now let $I = I_{2,\C}$. Since the complex special orthogonal group $\SOgroup(n,\C)$ is connected \cite{FultonHarris1991RepresentationTheory}, it follows that the product group
\[
G_\C = \SOgroup(n_1,\C) \times \cdots \times \SOgroup(n_d,\C)
\]
is connected as well. In particular, the orbit $\Vcal_\C(I_\C)$ is connected in the Euclidean topology. Hence we obtain the following corollary.

\begin{corollary}\label{cor: primary}
Both $I_{2,\C}$ and $I_2$ are primary, and the varieties $\Vcal_\C(I_{2,\C})$ and $\Vcal_\R(I_2)$ are irreducible.
\end{corollary}

\subsection{The Ideal \texorpdfstring{$I_2$}{I2} is Prime}\label{section:I2-prime}

We now apply Theorem~\ref{thm: primary-2-prime} to show that $I_2$ is prime. Since the previous subsection established that $I_2$ and $I_{2,\C}$ are primary, it remains to exhibit a maximal algebraically independent subset of variables and, for each remaining variable, a strongly square-free elimination polynomial in the sense of Theorem~\ref{thm: primary-2-prime}. 
The same construction works for all even numbers $p=2s>0$; thus, once $I_p$ is known to be primary, the argument below implies that $I_p$ is prime as well. 
As algebraic independence is preserved under complexification, we argue over~$\R$; all conclusions carry over to~$I_{p,\C}$. For brevity, write $\Vcal(I_p)=\Vcal_\R(I_p)$.

Let the tensor size be $\nbf=(n_1,\dots,n_d)$. Set
\begin{eqnarray}\label{eq:Jcal}
\Jcal_0 := \big\{ a\in[\nbf] : \big|\{i\in[d] : a_i \neq 1\}\big| \le 1 \big\}, \qquad
\Jcal := \Jcal_0 \setminus \{(1,\dots,1,n_d)\}.
\end{eqnarray}
We first show that $\Jcal$ is algebraically independent for every even $p=2s>0$.

\begin{proposition}\label{prop: alg.ind.set}
Let $p=2s>0$ and $\Jcal\subset[\nbf]$ be as above. Then $\R[\Jcal]\cap I_p=\{0\}$, i.e.\ $\Jcal$ is algebraically independent with respect to $I_p$.
\end{proposition}

\begin{proof}
By elimination/projection, $\pi_\Jcal(\Vcal(I_p)) \subseteq \Vcal(I_p\cap\R[\Jcal])$. We claim that $\pi_\Jcal(\Vcal(I_p))$ contains a nonempty Euclidean open set; hence it is Zariski dense in $\R^\Jcal$, so $\Vcal(I_p\cap\R[\Jcal])=\R^\Jcal$ and therefore $I_p\cap\R[\Jcal]=\{0\}$ by the Real Nullstellensatz.

Let $T=u^{(1)}\otimes\cdots\otimes u^{(d)}$ be rank one with $u^{(i)}\in\R^{n_i}$ and impose $\prod_{i=1}^d\|u^{(i)}\|_p=1$. Then
\[
\pi_\Jcal(\Vcal(I_p))=\Big\{\big(\prod_{i=1}^d u^{(i)}_{a_i}\big)_{a\in\Jcal} : \prod_i \|u^{(i)}\|_p=1\Big\}.
\]
Assume $u^{(i)}_{a_i}>0$ for the coordinates that appear. Writing $u^{(i)}_j=e^{v^{(i)}_j}$ and noting that $a_d\neq n_d$ for $a\in\Jcal$, we obtain the linear map
\[
L:\R^{n_1}\times\cdots\times\R^{\,n_d-1}\to\R^\Jcal,\qquad (v^{(1)},\dots,v^{(d)})\mapsto \Big(\sum_{i=1}^d v^{(i)}_{a_i}\Big)_{a\in\Jcal}.
\]
The rows of $L$ are linearly independent: for $a\neq(1,\dots,1)$, exactly one variable $v^{(i)}_j$ with $j>1$ appears only in the row indexed by $a$, and the row for $(1,\dots,1)$ involves the $v^{(i)}_1$’s. Hence $L$ has full rank, is open, and so is the componentwise exponential. Choosing
\[
U := (-\infty,0)^{n_1}\times\cdots\times(-\infty,0)^{n_{d-1}}\times \log\!\Big(\big\{u\in\R^{n_d-1}: u_j>0,\ \|u\|_p^p<\prod_{i<d} n_i^{-1}\big\}\Big),
\]
we can adjust $u^{(d)}_{n_d}$ to meet $\prod_i\|u^{(i)}\|_p=1$, hence $\exp\!\circ L(U)\subset \pi_\Jcal(\Vcal(I_p))$ is open. The claim follows.
\end{proof}

To prove that $\Jcal$ is maximal, we must, for every $a\in\Jcal^c$, produce a nonzero polynomial in $I_p\cap \R[\Jcal\cup\{a\}]$ which is strongly square-free in~$x_a$. Let $k=\R(\Jcal)$ denote the rational function field.

\begin{proposition}\label{prop: square-free-poly}
Let $p=2s>0$. For each $a\in\Jcal^c$ there exists $H_{p,a}\in I_p\cap \R[\Jcal\cup\{a\}]$ such that $H_{p,a}(x_a),\, H'_{p,a}(x_a)\in k[x_a]$ are coprime.
\end{proposition}

Before the proof, recall the rank-one binomials $x_ax_b-x_{a\vee b}x_{a\wedge b}$ for $a,b\in[\nbf]$. In particular, with $b=(1,\dots,1)$, one gets
\[
x_a x_{1\cdots1} \equiv x_{a_1\,1\cdots1}\,x_{1\,a_2\,\cdots a_d}\pmod{I_p},
\]
and iterating,
\begin{equation}\label{eq:x1dxa}
x_{1\cdots1}^{\,d-1}x_a \equiv g_a(x) := x_{a_1\,1\cdots1}\,x_{1\,a_2\,1\cdots1}\cdots x_{1\cdots1\,a_d}\pmod{I_p}.
\end{equation}

\begin{proof}[Proof of Proposition~\ref{prop: square-free-poly}]
We distinguish two cases. Throughout, write $f=g$ to mean $f\equiv g\pmod{I_p}$.

\emph{Case 1: $a\in\Ical_1:=\{a\in[\nbf]: |\{i: a_i\neq 1\}|\ge 2,\ a_d\neq n_d\}$.} By \eqref{eq:x1dxa},
\[
H_{p,a}(x_a):=x_{1\cdots1}^{\,d-1}x_a - g_a(x)\in I_p\cap \R[\Jcal][x_a].
\]
Then $H'_{p,a}(x_a)=x_{1\cdots1}^{\,d-1}\in k[x_a]$, and $H_{p,a},H'_{p,a}$ are coprime.

\emph{Case 2: $a_d=n_d$.} With $p=2s$, consider
\[
x_{1\cdots1}^{\,2s(d-1)}\!\Big(1-\sum_{b\in\Jcal}x_b^{2s}\Big)\prod_{i\ne d} x_{1\cdots1\,a_i\,1\cdots1}^{\,2s}
= x_{1\cdots1}^{\,2s(d-1)}\!\Big(\sum_{b\in\Jcal^c}x_b^{2s}\Big)\prod_{i\ne d} x_{1\cdots1\,a_i\,1\cdots1}^{\,2s}\in \R[\Jcal].
\]
Split the sum over $\Jcal^c$ into $\Ical_1\cup(\Jcal^c\setminus(\Ical_1\cup\{a\}))\cup\{a\}$. For $b\in\Ical_1$, \eqref{eq:x1dxa} rewrites $x_b$ in $\R[\Jcal]$; for $b\neq a$ with $b_d=n_d$, \eqref{eq:x1dxa} still applies, and the remaining factor $x_{1\cdots1\,n_d}$ can be eliminated using the reverse of \eqref{eq:x1dxa}:
\[
\prod_{i<d} x_{1\cdots1\,a_i\,1\cdots1}\cdot x_{1\cdots1\,n_d} = x_{1\cdots1}^{\,d-1}x_a \in \R[\Jcal\cup\{a\}].
\]
Collecting terms yields polynomials $f(x),g(x)\in \R[\Jcal]$, with $f$ a sum of squares and $g$ a nontrivial combination of sums of squares of different degree, such that
\[
H_{p,a}(x_a):= f(x)\,x_a^{2s} - g(x)\in I_p\cap \R[\Jcal\cup\{a\}].
\]
Then $H'_{p,a}(x_a)=2s\,f(x)\,x_a^{2s-1}$, and $f\neq0\neq g$ implies $\gcd(H_{p,a},H'_{p,a})=1$ in $k[x_a]$.
\end{proof}

\begin{example}
Let $p = 2$, the nuclear norm.
For $3\times3$ matrices (variables $x_{ij}$), the rank-one binomials are the $2\times2$ minors. One maximal algebraically independent index set is $\Jcal=\{(1,1),(1,2),(2,1),(3,1)\}$. Here $\Jcal^c=\{(2,2),(3,2)\}\cup\{(1,3),(2,3),(3,3)\}$ and $\Ical_1=\{(2,2),(3,2)\}$. For $a=(2,2)$, the minor gives
\[
H_a=x_{11}x_{22}-x_{12}x_{21}.
\]
For $a=(2,3)\notin\Ical_1$, compute
\[
x_{11}^2\big(1-x_{11}^2-x_{12}^2-x_{21}^2-x_{31}^2\big)x_{21}^2
= x_{11}^2\big[(x_{22}^2+x_{32}^2)+(x_{13}^2+x_{33}^2)+x_{23}^2\big]x_{21}^2,
\]
whose left side lies in $\R[\Jcal] = \R[x_{11},x_{12},x_{21},x_{31}]$. The terms with indices in $\Ical_1 = \{(2,2),(3,2)\}$ reduce via minors; the factor $x_{11}^2x_{33}^2=x_{13}^2x_{31}^2$ introduces $(1,3)\not\in \Jcal$, which we eliminate using the minor relation $x_{13}x_{21}=x_{11}x_{23}$. Collecting terms yields $H_a\in I_2\cap \R[\Jcal\cup\{a\}]$.
\end{example}

\begin{theorem}
Let $\Jcal$ be as in \eqref{eq:Jcal} and $I_2$ as in \eqref{eq: nuclear-ideal}. Then $I_2$ is prime, with
\[
\dim(I_2)=|\Jcal|=\sum_{i=1}^d n_i - d.
\]
\end{theorem}

\begin{proof}
By Proposition~\ref{prop: alg.ind.set} and Proposition~\ref{prop: square-free-poly}, the set $\Jcal$ is maximal algebraically independent and, for each $a\in\Jcal^c$, there exists a strongly square-free elimination polynomial as required. Since $I_2$ and $I_{2,\C}$ are primary (Corollary~\ref{cor: primary}), Theorem~\ref{thm: primary-2-prime} applies and yields that $I_2$ is prime. The dimension formula follows from $|\Jcal|$.
\end{proof}

\begin{remark}
For general even $p=2s>0$, if $I_p$ is primary, then the same argument shows that $I_p$ is prime with $\dim(I_p)=\sum_{i=1}^d n_i - d$.
\end{remark}

\subsection{The Ideal \texorpdfstring{$I_2$}{} is a Real Reduced Prime}\label{section:I2-real-reduced}

Marshall \cite{Marshall2008PositiveSOS} gave a useful criterion for checking if a prime ideal is real reduced.

\begin{lemma}[Single-Point Criterion]
Let $I\subset\R[X]$ be a prime ideal. Then $I$ is real radical if and only if there exists one non-singular real point $y\in\Vcal_\R(I)$.
\end{lemma}

The non-singularity is checked by the Jacobian matrix.  
Let $I=\langle f_1,\dots,f_k\rangle\subseteq\R[x_1,\dots,x_n]$ and $y\in\Vcal_\R(I)$.  
The Jacobian of $I$ at $y$ is
\[
\mathrm{Jac}_y(I)=\Big[\frac{\partial f_i}{\partial x_j}(y)\Big]_{i,j}.
\]
The point $y$ is non-singular if $\operatorname{rank}(\mathrm{Jac}_y(I)) = n - \dim(I)$.
It suffices if the direction $\operatorname{rank}(\mathrm{Jac}_y(I)) \geq n - \dim(I)$ is satisfied.

\begin{theorem}
For any tensor size $\nbf=(n_1,\dots,n_d)\in\N^d$, the ideal $I_2$ is a real reduced prime.
\end{theorem}

\begin{proof}
From Section~\ref{section:I2-prime}, we already know that $I_2$ is a prime ideal.  
Now we check the smoothness at the point $y=e_1\otimes\cdots\otimes e_1=(1,0,\dots,0)\in\Vcal_\R(I_2)$.  
Then the criterion above gives the claim.

We take the generators of $I_2$ from (\ref{eq: nuclear-ideal}).  
Define the index set
\[
\Jcal'=\Jcal\cup\{(1,\dots,1,n_d)\}\setminus\{(1,\dots,1)\},
\]
which means all indices with exactly one entry not equal to $1$.
For $a\notin\Jcal'$, we write $e_a:=e_{a_1}\otimes\cdots\otimes e_{a_d}$.
We show that all these $e_a$ appear in the row space of $\mathrm{Jac}_y(I_2)$.

If $a=(1,\dots,1)$, then the differential of $\sum_{b\in[\nbf]}x_b^2-1$ at $y$ gives the vector $e_{1\cdots1}$.  
If $a\notin\Jcal'$, then at least two entries of $a$ are not equal to $1$, for example $a_1,a_2\neq1$.  
Set $a'=(1,a_2,\dots,a_d)$.  
Then the binomial
\[
x_a x_{1\cdots1} - x_{a'}x_{a_1\,1\cdots1}\in I_0\subset I_2
\]
has derivative at $y$ equal to $e_a$.  
Therefore, every $e_a$ with $a\notin\Jcal'$ belongs to the row space of $\mathrm{Jac}_y(I_2)$.

We have
\[
\operatorname{rank}(\mathrm{Jac}_y(I_2)) \ge \#\{a\in[\nbf]: a\notin\Jcal'\} = \prod_{i=1}^d n_i - |\Jcal'|.
\]
From the previous section, $\dim(I_2)=|\Jcal|=\sum_{i=1}^d n_i - d$, and by definition $|\Jcal'|=|\Jcal|$.  
Hence
\[
\operatorname{rank}(\mathrm{Jac}_y(I_2)) = \prod_{i=1}^d n_i - \dim(I_2),
\]
so $y$ is non-singular.  
By the single-point criterion, $I_2$ is real radical, and since it is also prime, we conclude that $I_2$ is real reduced and prime.
\end{proof}

\begin{remark}
For other even $p=2s>0$, if $I_p$ is already known to be prime for a fixed tensor size, 
the same argument at the point $e_1\otimes\cdots\otimes e_1$ also proves that $I_p$ is real reduced.
\end{remark}

\textbf{Acknowledgements}:
The authors acknowledge funding by the Deutsche Forschungsgemeinschaft (DFG, German Research Foundation) - project number 442047500 - through the Collaborative Research Center ``Sparsity and Singular Structures'' (SFB 1481).

\bibliographystyle{alpha}
\bibliography{library}

\end{document}